\documentclass[11pt,a4paper]{amsart}

\usepackage[utf8]{inputenc}
\usepackage{amssymb,amsmath,amsbsy,amscd,amsfonts,amsthm,latexsym}
\usepackage[mathcal]{eucal}
\usepackage{times, euler}
\usepackage{eurosym}

\usepackage{stackrel, enumerate}
\usepackage{mathabx}

\usepackage[usenames]{color}

\usepackage[top=3cm,bottom=3cm,left=2.5cm,right=2.5cm,footskip=17mm]{geometry}

\usepackage{hyperref}

\usepackage[usenames]{color}

\usepackage{soul}

\usepackage{hyperref}

\usepackage{xy}
\xyoption{all}

\theoremstyle{plain}
    \newtheorem{theorem}[equation]{Theorem}
    \newtheorem{lemma}[equation]{Lemma}
    \newtheorem{corollary}[equation]{Corollary}
    \newtheorem{proposition}[equation]{Proposition}

\theoremstyle{definition}

    \newtheorem*{remark*}{Remark}
    \newtheorem*{remarks*}{Remarks}
    \newtheorem*{example*}{Example}
    \newtheorem*{definition*}{Definition}

    \newcommand{\C}{\mathbb{C}}
    
    \newcommand{\Z}{\mathbb{Z}}

    \renewcommand{\k}{\Bbbk}
    
   	\renewcommand{\phi}{\varphi}
	\let\epsilon\varepsilon

\newcommand{\germ}{\mathfrak}

\newcommand{\transpose}{\mathrm{t}}

\newcommand{\length}{\ell}
\newcommand{\partition}{P}

   \DeclareMathOperator{\Hom}{Hom}
    \DeclareMathOperator{\End}{End}

    \DeclareMathOperator{\Res}{Res}
    
    \DeclareMathOperator{\Ad}{Ad}

    \DeclareMathOperator{\GL}{GL}
    
    \DeclareMathOperator{\Rep}{Rep}

	\DeclareMathOperator{\opp}{opp}
	
	\DeclareMathOperator{\pind}{i}
	\DeclareMathOperator{\pres}{r}

	\DeclareMathOperator{\diag}{diag}

\keywords{{Harish-Chandra induction, principal series, general linear groups, finite commutative rings}}

\subjclass[2010]{Primary 20G05; Secondary 20C33, 20C15, 20C05.}

\begin{document}

\begin{abstract} We construct, for any finite commutative ring $R$, a family of representations of the general linear group $\GL_n(R)$ whose intertwining properties mirror those of the principal series for $\GL_n$ over a finite field.
\end{abstract}

\title{Principal series for general linear groups over finite \\ commutative rings}

\author{Tyrone Crisp}
\address[Corresponding author]{Department of Mathematics and Statistics, University of Maine, Orono, ME 04469-5752, USA}
  \email{tyrone.crisp@maine.edu}

\author{Ehud Meir}
\address{Institute of Mathematics, University of Aberdeen, Fraser Noble Building, Aberdeen AB24 3UE, UK}
  \email{meirehud@gmail.com}

\author{Uri Onn}
\address{Mathematical Sciences Institute, The Australian National University, Canberra, Australia}
  \email{uri.onn@anu.edu.au}

\date{Last updated \today}
\maketitle

\section{Introduction} Among the irreducible, complex representations of reductive groups over finite fields, the simplest to construct and to classify are the \emph{principal series}: those obtained by Harish-Chandra induction from a minimal Levi subgroup; see, for instance,  \cite{Howlett-Kilmoyer80}. In this paper we  use a generalisation of Harish-Chandra induction to construct a `principal series' of representations of the group $\GL_n(R)$, where $R$ is any finite commutative ring with identity. Our main results assert that the well-known intertwining relations among the principal series for $\GL_n$ over a finite field also hold for the representations that we construct. 

The study of the principal series for reductive groups over finite fields can be viewed as the first step in the program to understand all irreducible complex representations of such groups in terms of what Harish-Chandra called the `philosophy of cusp forms' \cite{Harish-Chandra,Springer_cusp}. This program has met with considerable success. The basic ideas appear already in Green's determination \cite{Green55} of the irreducible characters of $\GL_n(\k)$, where $\k$ is a finite field, and these ideas have since been developed and generalised to a very great extent; see \cite{DM} for an overview. 
 
The theory for groups over finite rings is in a far less advanced state. Most efforts so far have been directed toward groups over principal ideal rings: see {\cite{Hill_Jord, Hill_Reg, Hill_SSandCusp, Hill_nilp, Lusztig1, Lusztig2, Onn, Stasinski, Chen-Stasinski, Stasinski-Stevens, Krakovski-Onn-Singla, Stasinski-survey}}. By contrast, the results presented below are valid for all finite rings, and they depend on the algebraic properties of the base ring in only a very limited way. For instance, we give a uniform construction of a family of irreducible representations of $\GL_n(R)$ for all finite local rings $R$, and to our knowledge these are the first results obtained in this degree of generality. 

The present paper is part of a project whose aim is to extend the philosophy of cusp forms to reductive groups over finite rings. Our construction, which is a special case of a general induction procedure developed in \cite{CMO1},  extends in a  natural way to produce more general `Harish-Chandra series'. The analysis of the intertwining properties of these more general series seems, however, to be substantially more involved than the results for the principal series presented here. See \cite[Section 5]{CMO1} and \cite{CMO3} for some partial results in this more general setting.

\subsection*{Notation and definitions.} Let $R$ be a finite commutative ring with $1$.  Let $G=\GL_n(R)$, let $L\cong (R^\times)^n$ be the  subgroup of diagonal matrices in $G$, and let $U$ and $V$ be the upper-unipotent subgroup and the lower-unipotent subgroup, respectively, in $G$. We  write $G(R)$, $L(R)$, etc., when it is necessary to specify~$R$.

The ring $R$ decomposes as a direct product of local rings: $R\cong R_1\times \cdots \times R_m$, and this decomposition is unique up to permuting the factors  \cite[Theorem VI.2]{McDonald}. There is a corresponding decomposition $G(R)\cong G(R_1)\times \cdots \times G(R_m)$, and similarly for $L$, $U$, and $V$. 
If $R$ is a local ring then we let $N(R)$ be the subgroup of monomial matrices in $G(R)$, that is, products of permutation matrices with diagonal matrices. If $R$ is not local then we define $N(R)=N(R_1)\times\cdots\times N(R_m)$, where the $R_i$ are the local factors of $R$ as above. Let $W(R)=N(R)/L(R)$. It will be convenient to realise $W(R)$ as a subgroup of $G(R)$, as follows: if $R$ is local, then we identify $W(R)$ with the group of permutation matrices; and in the general case we identify $W(R)$ with the product of the permutation subgroups in $G(R)\cong G(R_1)\times\cdots\times G(R_m)$. Note that following Lemma~\ref{lem:local}, we will be able to assume without loss of generality that $R$ is a local ring.

If $\chi:L\to \GL(X)$ is a representation of~$L$ {on a complex} vector space $X$, and if $w\in W$, then we let $w^*\chi$ denote  the representation $\chi\circ {\Ad_w^{-1}}:L\to \GL(X)$.  We let $W_\chi=\{w\in W\ |\ w^*\chi\cong \chi\}$.

For each subgroup $H\subseteq G$ we let $e_H$ denote the idempotent in the complex group ring $\C[G]$ corresponding to the trivial character of $H$: $e_H=|H|^{-1}\sum_{h\in H} h$. Since $L$ normalises $U$ and $V$, the idempotents $e_U$ and $e_V$ commute with $\C[L]$ inside $\C[G]$. 

We consider the functors 
\[
\pind:\Rep(L)\to \Rep(G) \qquad X \mapsto \C[G]e_U e_V \otimes_{\C[L]} X
\]
\[
\pres:\Rep(G)\to \Rep(L) \qquad Y \mapsto e_U e_V\C[G]\otimes_{\C[G]} Y,
\]
where $\Rep(G)$ denotes the category of complex representations, identified in the usual way with the category of left $\C[G]$-modules. This is a special case of the construction defined in \cite[Section 2]{CMO1}, which generalises a definition due to Dat \cite{Dat_parahoric}. The functors $\pind$ and $\pres$ are two-sided adjoints to one another; see \cite[Theorem 2.15]{CMO1} for a proof of this and other basic properties. 

{
\begin{definition*}
Let us say that an irreducible representation of $G$ is in the \emph{principal series} if it is isomorphic to a subrepresentation of $\pind\chi$ for some representation $\chi$ of $L$.
\end{definition*}
}

\begin{example*}\label{ex:HC}
If $R$ is a field, then the map  $\C[G]e_U \xrightarrow{f\mapsto fe_V} \C[G]e_V$
is known to be an isomorphism of $\C[G]$-$\C[L]$ bimodules; see \cite[Theorem 2.4]{Howlett-Lehrer_HC}. It follows that the functors $\pind$ and $\pres$ are naturally isomorphic to the familiar functors of \emph{Harish-Chandra induction and restriction}, i.e.\ the functors of tensor product with the bimodules $\C[G]e_U$ and $e_U\C[G]$, respectively.  The same is not true if $R$ is not a product of fields. Revisiting an example from \cite{CMO1}, let $1_L$ denote the trivial   representation of $L$. Then   $\C[G]e_U\otimes_{\C[L]}1_L\cong \C[G/LU]$, with $G$ acting by permutations of $G/LU$. The number of distinct irreducible subrepresentations in $\C[G/LU]$ depends in a delicate way on the underlying ring $R$; for instance, for {$G=\GL_n(\Z/p^k\Z)$ with $k\geq 2$ and $n\geq 3$}, this number depends on both $p$ and $k$ {\cite{OPV}}. By contrast, it follows from Theorem \ref{thm:main2} below that for any $R$ the number of distinct irreducible subrepresentations of $\pind 1_L $ is equal to $\partition(n)^m$, where $G=\GL_n(R)$, $\partition$ is the partition function, and $m$ is the number of maximal ideals in $R$. 
\end{example*}

\begin{example*}\label{ex:regular}
Suppose that $R$ is a finite discrete valuation ring, with maximal ideal $\mathfrak{m}$ and residue field $\k$, and let $r$ be the largest integer such that $\mathfrak{m}^r\neq 0$. Reduction modulo $\mathfrak m^r$ gives rise to a group extension
\[
0 \to G_r\cong (M_n(\k),+) \to G(R) \to G(R/\mathfrak m^r) \to 0, 
\]
which one can use to study the representations of $G(R)$ via Clifford theory; see \cite{Hill_Jord}, for example. In \cite{Hill_Reg}, Hill identified a class of representations that are particularly amenable to this approach: an irreducible representation $\pi$ of $G(R)$ is called \emph{regular} if its restriction to $G_r$ contains a character whose stabiliser under the adjoint action of~$G(\k)$ has minimal dimension. Explicit constructions of all such representations are given in \cite{Stasinski-Stevens, Krakovski-Onn-Singla}. 

An application of \cite[Theorem 3.4]{CMO1} gives the following criterion for regularity of the induced representations $\pind\chi$: if $\chi$ is an irreducible representation of $L(R)$, then $\pind \chi $ is regular if and only if the restriction of $\chi$ to the subgroup $L(R)\cap G_r\cong \k^n$ has trivial stabiliser under the permutation action of~$S_n$. Moreover, the representations $\pind\chi$, for $\chi$ satisfying the above condition, account for all of the regular representations associated to the split semisimple classes in $M_n(\k)$.

{
For $n=2$, all of the principal series representations of $G(R)=\GL_2(R)$ can be described in terms of regular representations, as follows. Let $\chi:L\to \C^\times$ be an irreducible representation of $L$. If $\pind\chi$ is irreducible, then there is a character $\tau:R^\times \to \C^\times$, an integer $k$, and a regular representation $\pi$ of $G(R/\mathfrak{m}^k)$ associated to a split semisimple class in $M_n(\k)$ such that $\pind\chi$ is isomorphic to the representation $(\tau\circ\det)\otimes \pi$, where $\pi$ is pulled back to a representation of $G(R)$. If $\pind\chi$ is not irreducible, then there is a character $\tau:R^\times \to \C^\times$ such that $\pind\chi$ is isomorphic to the representation $(\tau\circ\det)\otimes (1_G\oplus \operatorname{St})$, where $1_G$ is the trivial representation, and $\operatorname{St}$ is the Steinberg representation of $G(\k)$  pulled back to $G(R)$. (These assertions follow easily from Theorem 3 and Lemma \ref{lem:chi_product}, below, and from \cite[Theorem 3.4]{CMO1}.) For $n\geq 3$ the relationship between the principal series and the regular representations becomes more complicated.
}
\end{example*}

\section{Main results}  We will show that the following  well-known properties of the Harish-Chandra functors are shared by the functors $\pind$ and $\pres$ for $R$ an arbitrary finite commutative ring. 

\begin{theorem}\label{thm:main1}
There is a natural isomorphism $\pres\pind\cong \bigoplus_{w\in W} w^*$ of functors on $\Rep(L)$. Consequently, if $\chi$ and $\sigma$ are irreducible   representations of $L$, then 
\[
\dim_{\C} \left(\Hom_G(\pind \chi,\pind\sigma)\right) = \#\{w\in W\ |\ w^*\chi=\sigma\}.
\]
\end{theorem}

When $\sigma=\chi$, we have the following more precise statement:

\begin{theorem}\label{thm:main2}
For each irreducible representation $\chi$ of $L$ one has $\End_G(\pind \chi) \cong \C[W_\chi]$ as algebras.
\end{theorem}
 
 {Theorems \ref{thm:main1} and \ref{thm:main2} readily imply} the following combinatorial formula for the number of principal series representations. Following \cite{Andrews}, we let $\partition_k(n)$ denote the number of multipartitions of $n$ with $k$ parts: i.e., the number of   $k$-tuples $(\lambda^{(1)},\ldots,\lambda^{(k)})$, where each $\lambda^{(i)}$ is a partition of some non-negative integer $n_i$, and $\sum_i n_i=n$.

\begin{corollary}\label{cor:counting} 
If $R$ is isomorphic to a product $R_1\times\cdots\times R_m$ of finite local rings, and for each $j$ we set $k_j=|R_j^\times|$, then the principal series of $\GL_n(R)$ contains precisely $\prod_{j} \partition_{k_j}(n)$ distinct isomorphism classes of irreducible representations. 
\end{corollary}

\begin{remarks*}\quad 
\begin{enumerate}[$\square$]  
\item In the case where $R$ is a field, Theorems \ref{thm:main1} and \ref{thm:main2} are essentially due to Green \cite{Green55}; see~\cite{Steinberg51} for the case $\chi=1_L$, and see \cite{Springer_GLn} for an exposition. Both of these results have been  generalised to arbitrary Harish-Chandra series for arbitrary reductive groups: see \cite{Harish-Chandra} and \cite{Howlett-Lehrer_Hecke}, respectively. 
\item Theorems \ref{thm:main1} and \ref{thm:main2} can be extended, using \cite[Theorem 2.15(5)]{CMO1}, to the setting of smooth representations of the profinite groups $G(\mathcal O)$, where $\mathcal O$ is the ring of integers in a nonarchimedean local field. 
\end{enumerate}
\end{remarks*}

\section{Proofs} The first step in the proof of the main results is to reduce to the case of local rings.

\begin{lemma}\label{lem:local}
If Theorems \ref{thm:main1} and \ref{thm:main2} and Corollary \ref{cor:counting}  are true for all finite commutative local rings, then they are true for all finite commutative rings.
\end{lemma}

\begin{proof}
Let $R$ be a  finite commutative ring, and write $R$ as a product of local rings $R_1\times \cdots\times R_m$. All of the groups and the representation categories in Theorems \ref{thm:main1} and \ref{thm:main2} and in Corollary \ref{cor:counting}   then decompose into products accordingly: $G(R)\cong G(R_1)\times \cdots \times G(R_m)$, $\Rep(G(R))\cong \Rep(G(R_1))\times \cdots \times \Rep(G(R_m))$, and so on. The bimodule $\C[G(R)]e_{U(R)}e_{V(R)}$ decomposes as the tensor product of the bimodules $\C[G(R_j)]e_{U(R_j)}e_{V(R_j)}$, and likewise for $e_{U(R)}e_{V(R)}\C[G(R)]$, so the functors $\pind$ and $\pres$ are compatible with the above decompositions. By definition, the group $W$ also decomposes compatibly. Thus Theorems \ref{thm:main1} and \ref{thm:main2} and Corollary \ref{cor:counting}   over $R$ follow immediately from the corresponding results over the local factors~$R_j$. 
\end{proof}

\textbf{Assume from now on that  $\boldsymbol R$ is a local ring.} Let $\germ m$ denote the maximal ideal of $R$, and let $\k$ denote the residue field $R/\germ m$. Recall that $W\cong S_n$ is then the group of permutation matrices in $G$. We write $\length$ for the word-length function on $W$ with respect to  the standard generating set $S=\{(1\  2), \dots,(n-1\ n)\}$.

The following proposition collects the group-theoretical ingredients of the proof of Theorem~\ref{thm:main1}. 

\begin{proposition}\label{lem:G} \quad 
\begin{enumerate}[\rm(a)]
\item The multiplication map $U\times L\times V\to G$ is injective.
\item The reduction-mod-$\germ m$ map $G(R)\to G(\k)$ is surjective.
\item For each subgroup $H$ of $G$, let $H_0$ denote the intersection of $H$ with the kernel $G_0$ of the above reduction homomorphism. Then the multiplication map $U_0\times L_0\times V_0\to G_0$ is a bijection, and the same is true for any ordering of the three factors.
\item For each $w\in W$ the multiplication maps
\[
 (U\cap U^w)\times (U\cap V^w) \to U \qquad \text{and}\qquad (V\cap U^w)\times (V\cap V^w)\to V
\]
are bijections, where $U^w=w^{-1}Uw$, etc.
\item $G$ is the disjoint union $G=\bigsqcup_{w\in W} G_w$, where $G_w=V{wL}UG_0$.
\item For each $r,t\in W$ with $\length(t)\leq \length(r)$ and $t\neq r$ one has $ULV\cap t^{-1}Ur=\emptyset$. 
\end{enumerate}
\end{proposition}

\begin{proof}
Parts (a), (b), (c) and (d) are well-known and easily verified. Part (e) follows immediately from the Bruhat decomposition of $G(\k)$ \cite[(65.4)]{CR2}.

In part (f) we may assume without loss of generality that $R$ is a field, since  $ULV\cap t^{-1}Ur$ is empty if its reduction modulo $\germ m$ is empty. Let ${w_0}$ denote the longest element $(1,2,\ldots,n)\mapsto (n,\ldots,2,1)$ of $W$, and write $B$ for the upper-triangular subgroup $LU$ of $G$. We will show that under the stated assumptions on $t$ and $r$ one has
\begin{equation}\label{eq:bruhat_claim} 
B {w_0} B\cap t^{-1}Br {w_0} B=\emptyset.
\end{equation}
Since $ULV {w_0} =UL {w_0} U=B {w_0} B$, while $t^{-1}Ur {w_0} \subseteq t^{-1}Br {w_0} B$, we see that \eqref{eq:bruhat_claim} implies that $ULV\cap t^{-1}Ur=\emptyset$.

To prove \eqref{eq:bruhat_claim} we recall (from, e.g., \cite[(65.10)]{CR2}) that $(B,N, W, S)$ is a BN-pair in $G$ (note that we are assuming $R$ to be a field). This means, among other things, that $sBx\subseteq BsxB \sqcup BxB$ for all $s\in S$ and all $x\in W$; and that $BxB\cap ByB=\emptyset$ for all $x \neq y\in W$. The proof of \eqref{eq:bruhat_claim} is by induction on $\length(t)$. If $\length(t)=0$, so that $t=1$, then $B {w_0} B\cap Br {w_0} B=\emptyset$ unless $r=1$. For the inductive step, write $t=st_1$ with $s\in S$ and  $\length(t_1)=\length(t)-1$. We then have
\[
B {w_0} B \cap t^{-1}Br {w_0} B \subseteq \left( B {w_0} B\cap t_1^{-1}Bsr {w_0} B\right) \sqcup \left(B {w_0} B\cap t_1^{-1} Br {w_0}  B\right).
\]
We have $\length(t_1)=\length(t)-1\leq \length(r)-1\leq \length(sr)$, and $st_1\neq r\implies t_1\neq sr$, so the first term in the union is empty by induction. We have $\length(t_1) <\length(t)\leq \length(r)$, and in particular $t_1\neq r$, so the second term in the union is also empty. This proves \eqref{eq:bruhat_claim}.
\end{proof}

We equip $\C[G]$ with the Hermitian inner product $\langle\ |\ \rangle$ for which the group elements $g\in G$ constitute an orthonormal basis; and with the conjugate-linear involution $*$ defined on basis elements by $g^*=g^{-1}$. The two structures are related by the identity $\langle abc | d\rangle = \langle b | a^*{d}c^*\rangle$ for all $a,b,c,d\in \C[G]$. An element $a\in \C[G]$ is called self-adjoint if $a=a^*$.

\begin{lemma}\label{lem:z}
Let $A$ denote the unital subalgebra of $\C[G]$ generated by the idempotents $e_U$ and $e_V$. There is a self-adjoint, invertible element $z$ in the centre of $A$ such that $z(e_U e_V)^2=e_U e_V$ and $z(e_V e_U)^2 =  e_V e_U$.
\end{lemma}

\begin{proof} 
This is true for any pair of orthogonal projections on a finite-dimensional Hilbert space: see \cite[Theorem 2]{Halmos}, for example.
\end{proof}

\begin{remark*}
If $R$ is a field then   \cite[Theorem 2.4]{Howlett-Lehrer_HC} implies that there is a \emph{unique} element $z$ as in Lemma \ref{lem:z}. This is not the case over a general ring.
\end{remark*}

\begin{lemma} \label{lem:VUVw}
For each $w\in W$ we have $e_V e_{U^w} e_{V^w} = e_V e_U e_{V^w}$.
\end{lemma}

\begin{proof}
It is clear that $e_V = e_V e_{(V\cap U^w)}$ and similarly that $e_U = e_{(U\cap U^w)} e_U$. Proposition~\ref{lem:G}(d) gives $e_{(V\cap U^w)}  e_{(U\cap U^w)}  = e_{U^w}$, and it follows that $e_V e_U = e_V e_{U^w} e_U$. The same reasoning gives $e_{ U^w} e_{V^w} = e_{U^w} e_U e_{V^w}$,
and so $e_V e_U e_{ V^w} = e_V e_{U^w} e_U e_{V^w} = e_V e_{U^w} e_{V^w}$.
\end{proof}

\begin{lemma}\label{lem:Dat_wU}
For each $w\in W$ the map 
\[ 
{\phi_w}:e_{U^w} e_{V^w} \C[G] \xrightarrow{ x\mapsto e_V x } e_V e_U \C[G] 
\]
is an isomorphism of $\C[L]$-$\C[G]$ bimodules.
\end{lemma}

\begin{proof}
 The following argument is taken from \cite[Lemme 2.9]{Dat_parahoric}.  The map {$\phi_w$} is well-defined, because 
\[
e_V e_{U^w} e_{V^w}\C[G] = e_V e_U e_{V^w} \C[G] \subseteq e_V e_U \C[G]
\]
by Lemma \ref{lem:VUVw}. The map $\phi_w$ is injective, because for each $f\in \C[G]$ we have 
\[
w^{-1} z w e_{U^w}e_{V^w}\left(e_V e_{U^w}e_{V^w}f\right) = z^w (e_{U^w}e_{V^w})^2 f =e_{U^w}e_{V^w} f
\]
where $z$ is as in Lemma \ref{lem:z}, and in the first equality we used that $V=(V\cap V^w)(V\cap U^w)$. The domain and target of $\phi_w$ are isomorphic as vector spaces: indeed, $e_V e_U\C[G]= w_0 w e_{U^w}e_{V^w}\C[G]$, where $w_0$ is the longest element of $W$. Since $\phi_w$ is injective it is thus also an isomorphism. 
\end{proof}

\noindent{For each subset $K\subseteq G$, we let $\C[K]$ denote the vector subspace of $\C[G]$ spanned by $K$.}

\begin{proposition}\label{lem:summands}
For each $w\in W$ the map 
\[
\Phi: \C[wL] \to e_U e_V \C[G_w] e_U e_V \qquad wl  \mapsto e_U e_V wl  e_U e_V 
\]
is an isomorphism of $\C[L]$-bimodules. 
\end{proposition}

\begin{proof}
$\Phi$ is clearly a bimodule map. Let us show that it is injective. For $h\in \C[L]$ we have 
\[
\Phi(wh) = e_U e_V e_{U^{w^{-1}}} e_{V^{w^{-1}}} wh.
\]
The maps 
\[ 
e_{U^{w^{-1}}} e_{V^{w^{-1}}} \C[G] \xrightarrow{ x\mapsto e_V x } e_V e_U \C[G] 
\]
and
\[
e_V e_U \C[G] \xrightarrow{x\mapsto e_U x} e_U e_V \C[G]
\]
are isomorphisms by Lemma \ref{lem:Dat_wU}, so we are left to prove that {the map}
\[
wh \mapsto e_{ U^{w^{-1}}} e_{ V^{w^{-1}}} wh = w e_U h e_V 
\]
is injective on $\C[wL]$. It is, because Proposition~\ref{lem:G}(a) implies that the cosets $Ul V$ are all disjoint as $l$ ranges over $L$. Thus $\Phi$ is injective.

To prove that $\Phi$ is surjective, first note that $G_w = V w L G_0 U$ because $G_0$ is normal in $G$. Since $e_V v=e_V$ and $ue_U=e_U$ for all $v\in V$ and $u\in U$, we find that $e_U e_V \C[G_w] e_U e_V$ is spanned by elements of the form $e_U e_V wlg e_U e_V$, where $l\in L$ and $g\in G_0$. We will show that each element of this form is in the image of $\Phi$. 

For each $x\in V^w$ we have 
\[
gx=x(x^{-1}gx) \in V^w G_0 = V^w (V_0^w L_0 U_0^w) = V^w L_0 U_0^w 
\]
by Proposition~\ref{lem:G}(c). Let $\alpha:V^w\to V^w$, $\beta:V^w\to L_0$ and $\gamma:V^w\to U_0^w$ be the (unique) functions satisfying $gx = \alpha(x)\beta(x)\gamma(x)$ for all $x\in V^w$. Writing $e_U=e_{U\cap V^w} e_{U\cap U^w}$ and $e_V=e_{V\cap U^w}e_{V\cap V^w}$, we then have 
\[
\begin{aligned}
e_V wlg e_U e_V & = e_V wlg e_{U\cap V^w} e_{U\cap U^w} e_{V\cap U^w} e_{V\cap V^w} \\
& = e_V wlg\left(|U\cap V^w|^{-1} \sum_{x\in U\cap V^w} x\right) e_{U^w}e_{V\cap V^w} \\
& = |U\cap V^w|^{-1} \sum_{x \in U\cap V^w} e_V wl\alpha(x)\beta(x)\gamma(x)e_{U^w} e_{V\cap V^w}.
\end{aligned}
\]
Since $\gamma(x)\in U^w$ we have $\gamma(x)e_{U^w}=e_{U^w}$ {for each $x\in U\cap V^w$}. Since $\alpha(x)\in V^w$ we have $wl\alpha(x)l^{-1}w^{-1}\in V$, and consequently $e_V wl\alpha(x)=e_Vwl$ {for each $x$}. Continuing the computation with the space-saving notation $h=|U\cap V^w|^{-1}\sum_{x\in U\cap V^w} l\beta(x)\in \C[L]$, we find that
\[
\begin{aligned}
e_V wlg e_U e_V & = e_V w  h e_{U^w} e_{V\cap V^w} = e_V e_{U^{w^{-1}} \cap V} w h e_{U^w}  e_{V\cap V^w} \\ 
& = e_V w h e_{U\cap V^w} e_{U\cap U^w} e_{V\cap U^w} e_{V\cap V^w} = e_V wh e_U e_V,
\end{aligned}
\]
and so $e_U e_V wlv e_U e_V = \Phi(wh)$.
\end{proof}

\begin{proposition}\label{lem:lin-ind}
The set $\{ e_U e_V wl e_U e_V\in \C[G]\ |\ w\in W,\ l\in L\}$ is linearly independent.
\end{proposition}

\begin{proof} 
We know from Proposition~\ref{lem:summands} that for each $w\in W$ the set $\{e_U e_V wl e_U e_V\ |\ l\in L\}$ is linearly independent. We must show that for different choices of $w$ these sets are independent from one another. 

Suppose we had elements $h_w\in \C[L]$, not all zero, with $\sum_{w\in W} e_U e_V wh_w e_U e_V =0$. Let $t\in W$ be an element of minimal length such that $h_t$ is nonzero. To compactify the notation we shall write $y =t^{-1}$. 

Let $z$ be as in Lemma \ref{lem:z}, and write $\zeta = y^{-1}z y$. Thus $\zeta$ is a  self-adjoint, invertible element of $\C[G]$ which commutes with $e_{U^{y}}$ and $e_{V^{y}}$ and which satisfies  $\zeta (e_{U^{y}} e_{V^{y}})^2= e_{U^{y}} e_{V^{y}}$. For each $r\in W$ with $r\neq t$ such that $h_r \ne 0$ we have
\[
\begin{aligned}
& \left\langle\left. \zeta^2 e_{U^{y}} (e_U e_V t h_t e_U e_V) \right| e_U e_V r h_r e_U e_V \right\rangle 
 =\left\langle\left. \zeta^2  e_{U^{y}} e_U e_V e_{U^{y}} e_{V^{y}} th_t \right| r h_r e_{U^r}e_{V^r}e_U e_V\right\rangle  \\
& =\left\langle\left. \zeta^2 e_{U^{y}} e_U e_V  e_{U^{y}} e_{V^{y}} t h_t  e_V e_U e_{V^r} e_{U^r} h_r^* \right| r \right\rangle 
 =\left\langle\left. \zeta^2 e_{U^{y}} e_U e_V  e_{U^{y}} e_{V^{y}} e_{U^{y}} e_{V^{ry}} e_{U^{ry}} t h_t h_r^*  \right| r \right\rangle \\
& =\left\langle\left. \zeta^2 e_{U^{y}} e_U e_{V^{y}} e_{U^{y}} e_{V^{y}} e_{U^{y}}e_{V^{y}} e_{U^{ry}} t h_t h_r^*  \right|r \right\rangle  
=\left\langle\left. \zeta^2 (e_{U^{y}} e_{V^{y}})^3 e_{U^{ry}} t h_t h_r^* \right| r\right\rangle \\
& =\left\langle\left. e_{U^{y}} e_{V^{y}} e_{U^{ry}} t h_t h_r^* \right| r\right\rangle  
  =\left\langle\left. t e_{U} h_t h_r^* e_V e_{U^r} \right| r\right\rangle   =\left\langle\left. e_U h_t h_r^* e_V \right| t^{-1} e_U r\right\rangle = 0.
\end{aligned}
\]
Here we have repeatedly used the equality $\langle abc| d\rangle = \langle b| a^*dc^*\rangle$; in the fourth step we used Lemma~\ref{lem:VUVw} to replace $e_Ue_V e_{U^y}$ with $e_U e_{V^y} e_{U^y}$ and to replace $e_{U^y}e_{V^{ry}} e_{U^{ry}}$ with $e_{U^y} e_{V^y} e_{U^{ry}}$; in the fifth step we used Proposition~\ref{lem:G}(d) to write $e_{U^y} e_U e_{V^y} = e_{U^y} e_{V^y}$; and in the final equality we used Proposition~\ref{lem:G}(f), which applies because of the minimality of $\length(t)$, and which implies that the functions $e_U h_t h_r^* e_V$ and $t^{-1} e_U r$ are supported on disjoint subsets of $G$ and are therefore orthogonal.

It follows from this that 
\[
\begin{aligned}
0 &=\Big\langle \zeta^2 e_{U^y} e_U e_V th_t e_U e_V \mid \sum_{w\in W} e_U e_V w h_w e_U e_V \Big\rangle \\
&= \Big\langle \zeta^2 e_{U^y} e_U e_V th_t e_U e_V \mid e_U e_V t h_t e_U e_V \Big\rangle \\ &= \Big\langle \zeta e_{U^y} e_U e_V th_t e_U e_V \mid \zeta e_{U^y} e_U e_V th_t e_U e_V\Big\rangle, 
\end{aligned}
\]
where the last equality holds because $\zeta$ is self-adjoint, $e_{U^y}$ is a self-adjoint idempotent, and $\zeta$ and $e_{U^y}$ commute. Thus $\zeta e_{U^y}e_U e_V th_t e_U e_V=0$. Since $\zeta$ is invertible, and left multiplication by $e_{U^y}$ is injective on $e_U e_V\C[G]$ (Lemma \ref{lem:Dat_wU}), we conclude that $e_U e_V th_t e_U e_V=0$. By Proposition~\ref{lem:summands} this implies that $h_t=0$, contradicting our choice of $t$ and completing the proof of the proposition.
\end{proof} 

\begin{proof}[Proof of Theorem \ref{thm:main1}]
The functor $\pres\pind$ is naturally isomorphic to the functor of tensor product (over $\C[L]$) with the $\C[L]$-bimodule $e_U e_V \C[G]e_U e_V$, while the functor $\bigoplus_{w\in W} w^*$ is naturally isomorphic to the tensor product with the bimodule $\C[W\ltimes L]$. Since $G=\bigsqcup G_w$ we have 
\[
e_U e_V \C[G] e_U e_V = \sum_{w\in W} e_U e_V \C[G_w] e_U e_V.
\]
Proposition~\ref{lem:summands} thus implies that the $\C[L]$-bimodule map 
\[
\C[W\ltimes L] = \bigoplus_{w\in W} \C[wL] \xrightarrow{h\mapsto e_U e_V h e_U e_V} e_U e_V \C[G]e_U e_V
\]
is surjective. Proposition~\ref{lem:lin-ind} implies that this map is injective, so it is an isomorphism of bimodules, and induces a natural isomorphism of functors $\pres\pind\cong \bigoplus w^*$. The formula for the intertwining number follows from this isomorphism and from the fact that $\pind$ and $\pres$ are adjoints.
\end{proof}

We now turn to the proof of Theorem \ref{thm:main2}. Every irreducible representation $\chi$ of the abelian group $L\cong (R^\times)^n$ has the form
\[
\chi_1\otimes\cdots\otimes \chi_n : \diag(r_1,\ldots,r_n)\mapsto \chi_1(r_1)\cdots \chi_n(r_n)
\]
where each $\chi_i$ is a linear character $R^\times\to \C^\times$. For each such $\chi$ we let $e_\chi = |L|^{-1}\sum_{l\in L}\chi(l)^{-1}l$ be the corresponding primitive central idempotent in $\C[L]$.

\begin{lemma}\label{lem:end_alg}
The algebra $\End_G(\pind\chi)$ is isomorphic to the subalgebra $e_\chi e_U e_V \C[G]e_U e_Ve_{\chi}$ of $\C[G]$. 
\end{lemma}

\begin{proof}
We have 
\[
\pind\chi \cong \C[G]e_U e_V\otimes_{\C[L]} \C[L]e_\chi \cong \C[G]e_U e_V e_\chi = \C[G] z e_U e_V e_\chi 
\]
where $z$ is as in Lemma \ref{lem:z}. Since $ze_U e_V$ and $e_\chi$ are commuting idempotents in $\C[G]$, their product $E=ze_U e_Ve_\chi$ is an idempotent and we have $\End_G(\C[G]E)\cong \left(E\C[G]E\right)^{\opp}$ via the action of $E\C[G]E$ on $\C[G]E$ by right multiplication. Now $E\C[G]E$ is a finite-dimensional complex semisimple algebra, so it is isomorphic to its opposite, and we have $E\C[G]E= e_\chi e_U e_V \C[G]e_U e_Ve_\chi$.
 \end{proof}

\begin{lemma}\label{lem:end_triv}
For the trivial representation $1_L$ of $L$ we have $\End_G(\pind 1_L)\cong \C[W]$ as algebras.
\end{lemma}

\begin{proof}
First suppose that $R$ is a field, so that the functor $\pind$ is isomorphic to the functor of Harish-Chandra induction. Then, as we noted above, $\pind 1_L$ is isomorphic to the permutation representation on $\C[G/LU]$, and the isomorphism $\End_G(\pind 1_L)\cong \C[W]$ is a special case of well-known results of Iwahori-Matsumoto and Tits (see \cite[\S 68]{CR2} for an exposition).

Now let $R$ be a local ring with residue field $\k$. The quotient map $R\to \k$ induces a surjective map of algebras 
\begin{equation}\label{eq:end_triv_red}
  e_{L(R)} e_{U(R)} e_{V(R)}\C[G(R)]e_{U(R)}e_{V(R)}e_{L(R)} \longrightarrow e_{L(\k)} e_{U(\k)} e_{V(\k)}\C[G(\k)]e_{U(\k)}e_{V(\k)}e_{L(\k)}.
\end{equation}
Theorem \ref{thm:main1} implies that the domain of \eqref{eq:end_triv_red} is isomorphic as a vector space to $\C[W]$, while we have just seen that the range of \eqref{eq:end_triv_red} is isomorphic as an algebra to $\C[W]$. Since \eqref{eq:end_triv_red} is surjective, it is an algebra isomorphism.
\end{proof}
 
\begin{remark*}
The isomorphism in Lemma \ref{lem:end_triv} is not canonical. One can trace through the various maps appearing in the proof to construct a set of Iwahori-Hecke generators of $e_L e_U e_V\C[G]e_U e_V e_L$, although this will depend on the choice of an element $z$ as in Lemma \ref{lem:z}.
\end{remark*}

{
\begin{lemma}\label{lem:chi_product}
Let $\chi = \chi_1\otimes \cdots \otimes\chi_n$ be an irreducible representation of $L$, let $\tau:R^\times\to \C^\times$ be a character of $R^\times$, and let $\chi' = \tau\chi_1\otimes \cdots \otimes \tau\chi_2$. Then $\pind\chi' \cong (\tau\circ\det)\otimes \pind\chi$.
\end{lemma}

\begin{proof}
The algebra automorphism
\begin{equation}\label{eq:tau_twist}
\C[G]\to \C[G], \qquad g\mapsto \tau(\det g)g 
\end{equation}
sends $e_{\chi'}$ to $e_\chi$, and fixes $e_U$ and $e_V$. Thus \eqref{eq:tau_twist} induces an isomorphism of $\C[G]$-modules
\[
\C[G]e_Ue_V \otimes_{\C[L]} \C_{\chi'} \xrightarrow{\cong} (\tau\circ\det)\otimes \C[G]e_U e_V \otimes_{\C[L]} \C_{\chi}.\qedhere
\] 
\end{proof}
}

\begin{lemma}\label{lem:chi_scalar}
If $\chi=\chi_1^{n}$ is a tensor-multiple of a single character  of $R^\times$, then $\End_G(\pind\chi)\cong \End_G(\pind 1_L)$ as algebras.
\end{lemma}

{
\begin{proof}
Lemma \ref{lem:chi_product} ensures that $\pind\chi\cong (\chi_1\circ\det)\otimes \pind 1_L$.
\end{proof}
}

\begin{lemma}\label{lem:iw}
For each $w\in W$ there is a natural isomorphism of functors $\pind\circ w^*\cong \pind$.
\end{lemma}

\begin{proof}
The functor $\pind\circ w^*$ is given by tensor product with the $\C[G]$-$\C[L]$ bimodule $\C[G]e_U e_V w = \C[G] e_{U^w} e_{V^w}$, while the functor $\pind$ is given by tensor product with $\C[G]e_U e_V$. These two bimodules are isomorphic, by Lemma \ref{lem:Dat_wU}.
\end{proof}

Lemma \ref{lem:iw} implies that in order to compute the intertwining algebra $\End_G\left(\pind \chi\right)$ for an arbitrary character $\chi$ of $L$ we may permute the factors $\chi_i$ so that $\chi$ takes the form
\begin{equation}\label{eq:chi}
\chi = \chi_1^{n_1}\otimes \chi_2^{n_2}\otimes\cdots\otimes \chi_k^{ n_k}\qquad \text{where}\qquad \chi_j\neq \chi_i \text{ unless }j=i.
\end{equation}
 (The exponents indicate tensor powers.) We then have $W_\chi \cong S_{n_1}\times \cdots \times S_{n_k}$.
 
In the next lemma we shall consider general linear groups of different sizes, and we shall accordingly embellish the notation with subscripts to indicate the size of the matrices involved: so, for example, $L_a$ denotes the diagonal subgroup in $G_a=\GL_a(R)$, and $\pind_a$ is a functor from $\Rep(L_a)$ to $\Rep(G_a)$.
 
\begin{lemma}\label{lem:pind-prime}
If $\chi$ is as in \eqref{eq:chi} then $\End_{G_n}(\pind_n\chi)\cong\bigotimes_{j=1}^k\End_{G_{n_j}}\big(\pind_{n_j} (\chi_j^{n_j}) \big)
$ as algebras. 
\end{lemma}

\begin{proof}
Let us write $L'$ for the block-diagonal subgroup $G_{n_1}\times \cdots\times G_{n_k} \subseteq G_n$, which contains as subgroups the groups $U'= U_{n_1}\times \cdots \times U_{n_k}$ and $V'= V_{n_1}\times \cdots \times V_{n_k}$. Let $U''$ be the subgroup of block-upper-unipotent matrices 
\[
U''=\left\{ \begin{bmatrix} 1_{n_1\times n_1} &   & \bigast \\
  & \ddots &   \\
0 &   & 1_{n_k\times n_k}\end{bmatrix} \in G_n\right\},
\]
and let $V''=(U'')^\transpose$ be the corresponding group of block-lower-unipotent matrices. 

Let $\pind':\Rep(L')\to \Rep(G_n)$ be the functor of tensor product with the $\C[G_n]$-$\C[L']$ bimodule $\C[G_n]e_{U''} e_{V''}$. The semidirect product decompositions $U=U_n=U'\ltimes U''$ and $V=V_n =V'\ltimes V''$ give equalities $e_U=e_{U'}e_{U''}$ and $e_V = e_{V'}e_{V''}$, and hence an isomorphism of $\C[G_n]$-$\C[L_n]$ bimodules
\[
\C[G_n]e_{U_n}e_{V_n} \cong \C[G_n]e_{U''}e_{V''}\otimes_{\C[L']} \C[L']e_{U'}e_{V'}.
\]
It follows that 
\[
\pind_n \chi \cong \pind'\left(\bigotimes_{j=1}^k \pind_{n_j}(\chi_j^{n_j})\right).
\]
Since $\pind'$ is a functor, we obtain {from this isomorphism} a map of algebras 
\begin{equation}\label{eq:pind-prime}
\pind':\bigotimes_{j=1}^k \End_{G_{n_j}}\big( \pind_{n_j}(\chi_n^{n_j})\big) \to \End_{G_n}(\pind_n\chi).
\end{equation}

Now, the $\C[L']$-bimodule map 
\[
\C[L']\to \C[G_n]e_{U''}e_{V''},\qquad h\mapsto he_{U''}e_{V''}
\]
is injective, because the multiplication map $L'\times U''\times V''\to G_n$ is one-to-one. It follows from this that the identity functor on $\Rep(L')$ is a subfunctor of $\Res^{G_n}_{L'}\circ \pind'$. Thus $\pind'$ is a faithful functor, and in particular the map \eqref{eq:pind-prime} is injective. Since the domain and the range of this map have the same dimension as complex vector spaces, by Theorem \ref{thm:main1}, we conclude that \eqref{eq:pind-prime} is an algebra isomorphism.
\end{proof}

\begin{proof}[Proof of Theorem \ref{thm:main2}]
Lemma \ref{lem:iw} allows us to assume that $\chi$ has the form \eqref{eq:chi}, and in this case we have algebra isomorphisms
\[
\begin{split}
\End_G(\pind\chi) &\xrightarrow[\cong]{\textrm{Lem.~\ref{lem:pind-prime}}} \bigotimes_{j} \End_{G_{n_j}}\big( \pind_{n_j}(\chi_j^{n_j})\big) \\
& \xrightarrow[\cong]{\textrm{Lem.~\ref{lem:chi_scalar}}} \bigotimes_j \End_{G_{n_j}}( \pind_{n_j} 1_{L_{n_j}}) \\
&\xrightarrow[\cong]{\textrm{Lem.~\ref{lem:end_triv}}} \bigotimes_j \C[S_{n_j}] \cong \C[W_\chi]. \qedhere
\end{split}
\]
\end{proof}

\begin{proof}[Proof of Corollary \ref{cor:counting}] 
Choose an ordering $\{\chi_1,\ldots,\chi_k\}$ of the character group $\widehat{R^\times}$. Lemma \ref{lem:iw} and the intertwining number formula in Theorem \ref{thm:main1} imply that for  each principal series representation $\pi$ of $\GL_n(R)$ there is a unique $k$-tuple of non-negative integers $n_1,\ldots,n_k$ having $\sum_i n_i=n$, such that $\pi$ embeds in $\pind(\chi_1^{n_1}\otimes\cdots\otimes \chi_k^{n_k})$. 

Theorem \ref{thm:main2} implies that the  number of distinct irreducible subrepresentations of $\pind(\chi_1^{n_1}\otimes\cdots\otimes \chi_k^{n_k})$ is equal to the number of distinct irreducible representations of $S_{n_1}\times\cdots\times S_{n_k}$. The latter number is equal  to the number of   $k$-tuples $(\lambda^{(1)},\ldots,\lambda^{(k)})$, where each $\lambda^{(i)}$ is a partition of $n_i$. Allowing the exponents $n_i$ to vary shows that the total number of principal series representations is equal to $\partition_k(n)$, as claimed.   
\end{proof}

\subsection*{Acknowledgments}  
The first and second authors were partly supported by the Danish National Research Foundation through the Centre for Symmetry and Deformation (DNRF92). 
The first author was also supported by fellowships from the Max Planck Institute for Mathematics in Bonn, and from the Radboud Excellence Initiative at Radboud University Nijmegen.
The second author was also supported by the Research Training Group 1670 ``Mathematics Inspired by String Theory and Quantum Field Theory.'' 
The third author acknowledges the support of the Israel Science Foundation and of the Australian Research Council.

\end{document}